\documentclass{amsart}

\usepackage{amsmath}
\usepackage{amscd}
\usepackage{amssymb}
 \usepackage{color}

\input xy
\xyoption{all}

\newcommand{\pd}{\partial}
\newcommand{\bC}{{\mathbb C}}

\newcommand{\bQ}{{\mathbb Q}}

\newcommand{\bX}{{\mathbb X}}
\newcommand{\bZ}{{\mathbb Z}}

\newcommand{\half}{\frac{1}{2}}

  \DeclareMathOperator{\ord}{ord}

\DeclareMathOperator{\res}{res}

\newtheorem{theorem}{Theorem}[section]
\newtheorem{theorem/definition}{Theorem/Definition}[section]

\newtheorem{proposition}{Proposition}[section]
\newtheorem{lemma}{Lemma}[section]

\newtheorem{corollary}{Corollary}[section]
\newtheorem{Conjecture}{Conjecture}

\theoremstyle{remark}
\newtheorem{remark}{Remark}[section]

\theoremstyle{definition}

\newcommand{\bml}{\begin{multline}}
\newcommand{\eml}{\end{multline}}

\newcommand{\bag}{\begin{align}}
\newcommand{\egn}{ \end{align}}

\newcommand{\be}{\begin{equation}}
\newcommand{\ee}{\end{equation}}
\newcommand{\bea}{\begin{eqnarray}}
\newcommand{\eea}{\end{eqnarray}}
\newcommand{\ben}{\begin{eqnarray*}}
\newcommand{\een}{\end{eqnarray*}}
\newcommand{\bet}{\begin{equation}
\begin{split}}
\newcommand{\eet}{\end{split}
\end{equation}}

\definecolor{yellow}{rgb}{1,1,0}
\definecolor{orange}{rgb}{1,.7,0}
\definecolor{red}{rgb}{1,0,0} \definecolor{blue}{rgb}{0,0,1}
\definecolor{white}{rgb}{1,1,1}

\definecolor{A}{rgb}{.75,1,.75}

\begin{document}

\title
{Some Integrality Properties in Local Mirror Symmetry}
\author{Jian Zhou}
\address{Department of Mathematical Sciences\\Tsinghua University\\Beijing, 100084, China}
\email{jzhou@math.tsinghua.edu.cn}

\begin{abstract}
We prove some integrality properties of the open-closed mirror maps, inverse open-closed mirror maps and mirror curves
of some local Calabi-Yau geometries.
\end{abstract}


\maketitle

\section{Introduction}

In this paper we will establish some integrality properties
of some geometric objects in the study of local mirror symmetry,
such as local open-closed mirror map,
its inverse map and the equation of the mirror curve.
Mirror symmetry of Calabi-Yau $3$-folds relates two kinds of moduli spaces.
Given a compact Calabi-Yau $3$-folds $M$,
one can consider some special coordinates on its complexified K\"ahlar moduli space
and some special coordinates on the moduli space of complex structures on its mirror manifold $W$.
The mirror map is a map between these two moduli spaces
relating the relevant special coordinates.
It can be given by solutions of some Picard-Fuchs equations.
An amazing property of the mirror map is that its Taylor coefficients are rational integers.
(A close analogy is that modular functions or modular forms on the moduli space of elliptic curves
often have integral Fourier coefficients.)
This was first established by Lian and Yau \cite{Lia-Yau} for
some hypersurfaces in projective spaces using Dwork's Lemma,
Dwork's theory of $p$-adic hypergeometric series, and
$p$-adic Gamma function.
Their  method has been used  by other authors
to establish Lian-Yau integrality in more general cases,
see Zudilin \cite{Zud}, Krattenthaler-Rivoal \cite{Kra-Riv, Kra-Riv2}
and Delaygue \cite{Del}.

Local mirror symmetry \cite{CKYZ} is the extension of the mirror symmetry
to the case of some open Calabi-Yau $3$-folds.
In this case the Picard-Fuchs equations have the constant function $1$ as a trivial solution.
It follows that the local mirror map takes a much simpler form than
in the case of closed Calabi-Yau $3$-folds.
One can also consider the open-closed mirror maps of some noncompact Calabi-Yau
$3$-folds by extending the Picard-Fuchs system \cite{Ler-May}.
It turns out that the integrality of the local open-closed mirror maps can be reduced to the congruence properties of
the multinomial coefficients by Dwork's Lemma.
In this paper
we will verify directly the integrality properties of these multinomial numbers.
The integrality of the inverse local open-closed mirror maps follows easily from the integrality of the local
open-closed mirror maps.

In a separate  paper \cite{Zho},
we will make some observations on the integrality properties of the open-closed mirror maps
in the case of compact Calabi-Yau manifolds.

One can also obtain from the extended Picard-Fuch system corresponds
the superpotential which encodes the disc invariants.
According to the conjecture \cite{Aga-Vaf, Aga-Kle-Vaf} disc invariants
determines the mirror curve of the local Calabi-Yau geometry.
We will show that in the case of disc invariants with an outer brane,
the equation for the mirror curve has some integrality property.
In the case of an inner brane,
one will verify the integrality property of a related map,
whose geometric meaning will be clarified in a forthcoming work \cite{Zho2}.

The rest of this paper is arranged as follows.
In Section 2 we prove some inequalities related to $p$-adic expressions
of numbers.
They are used to establish some congruence properties of some multinomial numbers in Section 3,
which are applied in Section 4 to prove the integrality of some power series by Dwork's Lemma.
In Section 5 and Section 6 we apply these integrality properties to the study of local mirror symmetry
of some open Calabi-Yau $n$-folds.

\vspace{.1in}
{\em Acknowledgements}.
This research is partly supported by NSFC grants (10425101 and 10631050)
and a 973 project grant NKBRPC (2006cB805905).
The author thanks Professor Christian Krattenthaler for bringing \cite{Good} and \cite{Del} to his attentions.
He also thanks Hanxiong Zhang and Yinhua Ai for communicating to him the proofs of Proposition \ref{prop:Conj2}
and Proposition \ref{prop:Conj3} respectively.

\section{Some Inequalities for $p$-Adic Expressions}

Let $p$ be a prime.
Given a positive rational integer $n = a_0 + a_1 p + \cdots + a_s p^s$,
$0 \leq a_0, \dots, a_s < p$,
define
\be
S_p(n) = a_0 + a_1 + \cdots + a_s.
\ee
Let $0 \leq i \leq s$ be the smallest number such that $a_i \neq 0$,
define
\be
\ord_p (n) = i.
\ee
It is clear that
\be
S_p(p^rn) = S_p(n), \qquad \ord_p(p^rn) = \ord_p(n) + r.
\ee
For a rational number $\frac{a}{b} \in \bQ$,
set $\ord_p (\frac{a}{b}) = \ord_p(a) - \ord_p(b)$.
Then $\frac{a}{b} \in \bZ_p$ iff $\ord_p (\frac{a}{b}) \geq 0$.

The following is well-known and will be used repeatedly below \cite{Kob}:
\be
\ord_p(n!) = \frac{n-S_p(n)}{p-1}.
\ee
Applying this formula to $n =p^ra$ and $n/p$, where $r \geq 1$ and $a$ is a rational integer,
one gets:
\be \label{eqn:Order}
\ord_p (\frac{(p^ra)!}{(p^{r-1}a)!})
= p^{r-1}a.
\ee
Furthermore,
it is not hard to see that
\be \label{eqn:Remainder}
p^{-p^{r-1}a} \frac{(p^ra)!}{(p^{r-1}a)!} \equiv (\prod_{\substack{1 \leq j < p^r \\(p,n) = 1}} j)^a \pmod{p^r}.
\ee
A well-known fact is that \cite{Kob}:
\be
\prod_{\substack{1 \leq j < p^r \\(p,n) = 1}} j
\equiv
\begin{cases}
-1 \pmod{p^r}, & \text{when $p$ is odd}, \\
-1 \pmod{p^r}, & \text{when $p=2$ and $r = 2$}, \\
1 \pmod{p^r}, & \text{when $p=2$ and $r \neq 2$}.
\end{cases}
\ee
We will not use this fact below.

\begin{lemma}
For a positive rational integer $n$,
the following inequality holds:
\be
n - S_p(n) \geq (p-1) \ord_p(n).
\ee
\end{lemma}

\begin{proof}
Write $n=a_r p^r + \cdot + a_s p^s$, $r \leq s$, $0 \leq a_r, \dots, a_s < p$, $a_r >0$.
Then we have
\ben
n -S_p(n) & = & a_r (p^r -1) + \cdot + a_s (p^s -1) \\
&= &(p-1) [a_r (p^{r-1} + \cdots + 1) + \cdots + a_s (p^{s-1} + \cdots + 1)] \\
&\geq & a_r (p-1) r \geq (p-1) r = (p-1) \ord_p(n).
\een

\end{proof}

\subsection{Additive inequalities}
\begin{lemma}
For positive integers $a$ and $b$,
the following inequality holds for all prime $p$:
\be
S_p(a) + S_p(b) \geq S_p(a+b) + (p-1) \cdot
(\ord_p (a+b) - \min \{ \ord_p(a), \ord_p(b) \}  ).
\ee
\end{lemma}

\begin{proof}
Without loss of generality,
assume $\ord_p(b) \geq \ord_p(a) = 0$.
Write $a= a_0 + a_1 p + a_2 p + \cdots + a_k p^k$,
$b = b_0 + b_1 p + \cdots + b_k p^k$,
and $a+b = c_r p^r + c_{r+1} p^{r+1} + \cdots + c_{k+1} p^{k+1}$.
When $r = 0$,
because
\ben
&& (a_0 + a_1 p+ \cdots + a_k p^k)+(b_0 + b_1 p + \cdots + b_k p^k)
= c_0 + c_1p+ \cdots + c_{k+1} p^{k+1},
\een
it is easy to see that
$$(a_0 + \cdots + a_k) + (b_0 + \cdots + b_k) \geq c_0 + \cdots + c_{k+1}.$$
Now suppose that $r > 0$,
from
\be
(a_0 + a_1 p  + \cdots + a_k p^k)+(b_0 + b_1 p + \cdots + b_k p^k)
= c_rp^r + c_{r+1}p^{r+1}+ \cdots + c_{k+1} p^{k+1},
\ee
we see that
\be
b_0 = p-a_0, b_1 = p-a_1 - 1, \dots, b_{r-1} = p- a_{r-1} - 1, c_r =a_r + b_r+1 \pmod{p}.
\ee
Therefore,
$$S_p(a) + S_p(b) = r (p -1)+ 1+ \sum_{i=r}^k (a_i+b_i)
\geq r (p - 1) + \sum_{i=r}^{k+1} c_i.$$
\end{proof}

By induction one can also prove the following:

\begin{lemma} \label{lm:Additive2}
For positive integers $a_1, \dots, a_n$,
the following inequality holds for all prime $p$:
\be
\sum_{i=1}^n S_p(a_i)  \geq S_p(\sum_{i=1}^n a_i) + (p-1) \cdot
(\ord_p (\sum_{i=1}^n a_i) - \min \{ \ord_p(a_1),\dots, \ord_p(a_n) \}  ).
\ee
\end{lemma}

\subsection{Multiplicative inequalities}

\begin{lemma} \label{lm:Multiplicative}
The following inequality holds for any positive integers $m$, $n$ and any prime $p$:
\be \label{eqn:Multiplicative}
S_p(mn) \leq S_p(m) \cdot S_p(n).
\ee
\end{lemma}

\begin{proof}
Write $m = b_0 + b_1 p + \cdots +b_t p^t$, $n =a_0 + a_1 p \cdots + a_s p^s$,
where $0 \leq b_0, \dots, b_t$, $a_0, \dots, a_s < p$.
It is easy to see that for any positive integers $a, b$,
\be
S_p(a+b) \leq S_p(a) + S_p(b),
\ee
and when $a,b$ are positive integers smaller than $p$,
\be
S_p(ab) \leq ab.
\ee
Therefore,
we have
\ben
S_p(mn) & = & S_p(\sum_k (\sum_{i+j=k} a_ib_j )p^k)
\leq \sum_k \sum_{i+j =k} S_p(a_ib_j) \\
& \leq & \sum_k \sum_{i+j=k} a_i \cdot b_j
= S_p(m) \cdot S_p(n).
\een
\end{proof}

\begin{corollary}
The following inequality holds for any positive integers $m$, $n$ and any prime $p$:
\be
m S_p(n) - S_p(mn) + S_p(m) - m \geq (m- S_p(m)) \cdot (S_p(n) -1).
\ee
In particular,
when $S_p(n) > 1$,
i.e.,
$n \neq p^r$ for some $r \geq 0$,
\be
m S_p(n) - S_p(mn) + S_p(m) - m \geq m - S_p(m).
\ee
\end{corollary}

\begin{proof}
Because we have $ m \geq S_p(m)$ and $S_p(n) \geq 1$,
by the inequality \eqref{eqn:Multiplicative} we have
\ben
&& m S_p(n) - S_p(mn) + S_p(m) - m \\
& \geq & m S_p(n) - S_p(m)\cdot S_p(n) + S_p(m) - m \\
& = & (m- S_p(m)) \cdot (S_p(n) -1).
\een
Then the above equalities are obvious.
\end{proof}

\section{Congruence Properties  of Some Multinomial Numbers}

We will use the inequalities obtained in last section
to derive some congruence properties of
some multinomial numbers.

\subsection{Congruence properties of $\binom{k_1+ \cdots+ k_n}{k_1, \dots, k_n}$}

\begin{proposition} \label{prop:1}
For rational integers $k_1, \dots, k_n \geq 0$ such that $\sum_{i=1}^n k_i > 0$,
and $(p, k_1, \dots, k_n) = 1$,
we have
\be
\frac{1}{\sum_{i=1}^n k_i} \binom{\sum_{i=1}^n k_i}{k_1, \dots, k_n} \in \bZ_p.
\ee
\end{proposition}

\begin{proof}
Because $(p, k_1, \dots, k_n) =1$,
not all $k_i$'s are divisible by $p$,
therefore,
$$\min\{\ord_p(k_1), \dots, \ord_p(k_n)\}=0.$$
So we have
\ben
&& \ord_p\biggl( \frac{1}{\sum_{i=1}^n k_i} \binom{\sum_{i=1}^n k_i}{k_1, \dots, k_n} \biggr)
= \ord_p \biggl( \frac{1}{\sum_{i=1}^n k_i} \frac{(\sum_{i=1}^n k_i)!}{\prod_{i=1}^n k_i!} \biggr) \\
& = & \ord_p ((\sum_{i=1}^n k_i)!) - \sum_{i=1}^n \ord_p (k_i!)  - \ord_p (\sum_{i=1}^n k_i) \\
& = & \frac{\sum_{i=1}^n k_i-S_p(\sum_{i=1}^n k_i)}{p-1} -\sum_{i=1}^n \frac{k_i-S_p(k_i)}{p-1}
- \ord_p(\sum_{i=1}^n k_i) \\
& = & \frac{1}{p-1} (\sum_{i=1}^n S_p(k_i)- S_p(\sum_{i=1}^n k_i)) - \ord_p(\sum_{i=1}^n k_i)
\geq 0.
\een
In the last inequality we have used Lemma \ref{lm:Additive2}.
\end{proof}

\begin{proposition} \label{prop:1.1}
For rational integers $m, k_1, \dots, k_n > 0$ such that $(p, m) = 1$,
we have
\be
\frac{1}{m} \binom{\sum_{i=1}^n k_i m}{k_1m, \dots, k_nm} \in \bZ_p.
\ee
\end{proposition}

\begin{proof}
Because $(p, m) =1$,
we have $\ord_p(m) = 0$.
So we have
\ben
&& \ord_p\biggl( \frac{1}{m} \binom{\sum_{i=1}^n k_i }{k_1m, \dots, k_nm} \biggr)
= \ord_p \biggl( \frac{1}{m} \frac{(\sum_{i=1}^n k_im)!}{\prod_{i=1}^n (k_im)!} \biggr) \\
& = & \ord_p ((\sum_{i=1}^n k_i m)!) - \sum_{i=1}^n \ord_p ((k_i m)!)  - \ord_p (m) \\
& = & \frac{\sum_{i=1}^n k_i m-S_p(\sum_{i=1}^n k_im)}{p-1} -\sum_{i=1}^n \frac{k_im-S_p(k_im)}{p-1} \\
& = & \frac{1}{p-1} (\sum_{i=1}^n S_p(k_im)- S_p(\sum_{i=1}^n k_im)) \\
& \geq & \ord_p (\sum_{i=1}^n k_i m) - \min \{\ord_p(k_im)\}_{i=1}^n \geq 0.
\een
In the last inequality we have used Lemma \ref{lm:Additive2}.
\end{proof}

By modifying the proof slightly,
we also have

\begin{proposition} \label{prop:1.2}
For rational integers $m, k_1, \dots, k_n > 0$ such that $(p, m) = 1$ and $(p, k_1, \dots, k_n) = 1$,
we have
\be
\frac{1}{\sum_{i=1}^n k_i m} \binom{\sum_{i=1}^n k_i m}{k_1m, \dots, k_nm} \in \bZ_p.
\ee
\end{proposition}

Similarly,

\begin{proposition} \label{prop:1.3}
For rational integers $k_1, \dots, k_n > 0$ such that $(k_1, \dots, k_n) =1$,
we have
\be
\frac{1}{\sum_{i=1}^n k_i} \binom{\sum_{i=1}^n k_i m}{k_1m, \dots, k_nm} \in \bZ.
\ee
\end{proposition}

\begin{proposition}
For rational integers $k_1, \dots, k_n \geq 0$ such that $\sum_{i=1}^n k_i > 0$,
if all $k_i$'s are divisible by a prime $p$,
then the following inequality holds:
\bea
&& \ord_p \biggl( \frac{(\sum_{i=1}^n k_i)!}{(\sum_{i=1}^n k_i/p)!}
- \prod_{i=1}^n \frac{k_i!}{(k_i/p)!} \biggr)
\geq \sum_{i=1}^n k_i/p + \min \{\ord_p(k_1), \dots, \ord_p(k_n)\}.
\eea
\end{proposition}

\begin{proof}
First by \eqref{eqn:Order} we have
\ben
&& \ord_p \biggl(\frac{(\sum_{i=1}^n k_i)!}{(\sum_{i=1}^n k_i/p)!}   \biggr)
= \sum_{i=1}^n k_i/p, \\
&& \ord_p \biggl(  \prod_{i=1}^n \frac{k_i!}{(k_i/p)!} \biggr)
= \sum_{i=1}^n k_i/p.
\een
Secondly,
assume $\ord_p(k_1) \geq \cdots \geq \ord_p(k_n) = r$.
By \eqref{eqn:Remainder} we have
\ben
&& p^{-\sum_{i=1}^n k_i/p} \frac{(\sum_{i=1}^n k_i)!}{(\sum_{i=1}^n k_i/p)!}
\equiv (\prod_{1 \leq j \leq p^r} j)^{\sum_{i=1}^n k_i/p^r} \pmod{p^r}, \\
&& p^{-\sum_{i=1}^n k_i/p} \prod_{i=1}^n \frac{(k_i)!}{(k_i/p)!}
\equiv (\prod_{1 \leq j \leq p^r} j)^{\sum_{i=1}^n k_i/p^r} \pmod{p^r}.
\een
Therefore,
the proof is complete.
\end{proof}

\begin{proposition} \label{prop:3}
For rational integers $k_1, \dots, k_n \geq 0$ such that $\sum_{i=1}^n k_i > 0$,
if all $k_i$'s are divisible by a prime $p$,
then  we have
\be
\frac{1}{\sum_{i=1}^n k_i} \biggl( \binom{\sum_{i=1}^n k_i}{k_1, \dots, k_n}
-  \binom{\sum_{i=1}^n k_i/p}{k_1/p, \dots, k_n/p} \biggr) \in \bZ_p.
\ee
\end{proposition}

\begin{proof}
We have
\ben
&& \ord_p \biggl( \frac{1}{\sum_{i=1}^n k_i} \biggl( \binom{\sum_{i=1}^n k_i}{k_1,\dots, k_n}
-  \binom{\sum_{i=1}^n k_i/p}{k_1/p, \dots, k_n/p} \biggr)  \biggr)\\
& = & \ord_p \biggl(\frac{1}{\sum_{i=1}^n k_i} \frac{((\sum_{i=1}^n k_i/p)!}{\prod_{i=1}^n k_i!}
 \biggl( \frac{(\sum_{i=1}^n k_i)!}{(\sum_{i=1}^n k_i/p)!} - \prod_{i=1}^n \frac{k_i!}{(k_i/p)!} \biggr) \biggr) \\
& \geq & \sum_{i=1}^n k_i/p + \min \{\ord_p(k_1), \dots, \ord_p(k_n)\} \\
&& + \frac{\sum_{i=1}^n k_i/p-S_p(\sum_{i=1}^n k_i/p)}{p-1}
- \sum_{i=1}^n \frac{k_i-S_p(k_i)}{p-1}- \ord_p (\sum_{i=1}^n k_i) \\
& = & \frac{1}{p-1} (\sum_{i=1}^n S_p(k_i)- S_p(\sum_{i=1}^n k_i)) \\
&& - (\ord_p(\sum_{i=1}^n k_i) - + \min \{\ord_p(k_1), \dots, \ord_p(k_n)\}) \geq 0.
\een
In the last equality we have used Lemma \ref{lm:Additive2}.
\end{proof}

As a special case,
we have the following

\begin{proposition} \label{prop:3.1}
For rational integers $m, k_1, \dots, k_n > 0$ such that $p|m$,
we have
\be
\frac{1}{\sum_{i=1}^n k_im} \biggl( \binom{\sum_{i=1}^n k_im}{k_1m, \dots, k_nm}
-  \binom{\sum_{i=1}^n k_im/p}{k_1m/p, \dots, k_nm/p} \biggr) \in \bZ_p.
\ee
\end{proposition}

\subsection{Congruence properties of $\frac{1}{m!}\binom{mn}{n, \dots, n}$}

\begin{proposition} \label{prop:4}
For any positive integers $m$ and $n$,
we have
\be
\frac{1}{m!} \binom{mn}{n, \dots, n} \in \bZ_p.
\ee
Furthermore,
when $S_p(n) > 1$,
\be
\frac{1}{m \cdot m!} \binom{mn}{n, \dots, n} \in \bZ_p.
\ee
\end{proposition}

\begin{proof}
We have
\ben
&& \ord_p\biggl(\frac{1}{m!} \binom{mn}{n, \dots, n} \biggr)
= \ord_p \biggl( \frac{(mn)!}{m! (n!)^m} \biggr) \\
& = & \ord_p ((mn)!) - \ord_p (m!) - m \ord_p (n!) \\
& = & \frac{mn-S_p(mn)}{p-1} - \frac{m-S_p(m)}{p-1} - m \cdot \frac{n-S_p(n)}{p-1} \\
& = & \frac{1}{p-1} (m S_p(n) - S_p(mn) + S_p(m) - m) \\
& \geq & \frac{1}{p-1} (m- S_p(m)) \cdot (S_p(n) -1).
\een
When $S_p(n) > 1$,
\ben
&& \ord_p\biggl(\frac{1}{m \cdot m!} \binom{mn}{n, \dots, n} \biggr)
 \geq \frac{1}{p-1}(m -S_p(m)) - \ord_p(m) \geq 0.
\een
\end{proof}

\begin{proposition}
Let $m$ and $n$ be two positive numbers,
$p$ a prime,
$n = p^r a$, $(p, a)=1$, $r \geq 1$.
Then the following inequality holds:
\be
\ord_p \biggl( \frac{(mp^r a)!}{(mp^{r-1} a)!}  - \biggl(\frac{(p^ra)!}{(p^{r-1}a)!} \biggr)^m \biggr)
\geq mp^{r-1}a + r.
\ee
\end{proposition}

\begin{proof}
By \eqref{eqn:Order} we have
\ben
&& \ord_p \biggl( \frac{(mp^r a)!}{(mp^{r-1} a)!}  \biggr)
= mp^{r-1}a, \\
&& \ord_p \biggl(\frac{(p^ra)!}{(p^{r-1}a)!} \biggr)^m
= mp^{r-1}a.
\een
By \eqref{eqn:Remainder},
we have
\ben
&& p^{-p^{r-1}a} \frac{(p^ra)!}{(p^{r-1}a)!}
\equiv (\prod_{\substack{0 < j < p^r \\ (j,p) = 1}} j)^a \pmod{p^r}, \\
&& p^{-mp^{r-1}a}  \frac{(mp^r a)!}{(mp^{r-1} a)!}
\equiv (\prod_{\substack{0 < j < p^r \\ (j,p) = 1}} j)^{ma} \pmod{p^r}.
\een
Therefore the proof is complete.
\end{proof}

Base on some numerical evidence obtained by Maple calculations,
we make the following:

\begin{Conjecture}
Let $p$ be a prime, and let $a$ be a positive integer such that $(p, a) = 1$,
$r \geq 1$ an integer.
Then we have
\be
\ord_p \biggl( \frac{(mp^r a)!}{(mp^{r-1} a)!}  - \biggl(\frac{(p^ra)!}{(p^{r-1}a)!} \biggr)^m \biggr)
= \begin{cases}
mp^{r-1}a + 3r - 1, & p =3, \\
 mp^{r-1}a + 3r, & p > 3.
\end{cases}
\ee

\end{Conjecture}

\begin{proposition} \label{prop:6}
For any positive integers $m$ and $n$,
if $n$ is divisible by a prime $p$,
then we have
\be
\frac{1}{m!\cdot n} \biggl( \binom{m n}{n, \dots, n}
-  \binom{m n/p}{n/p, \dots, n/p} \biggr) \in \bZ_p.
\ee
\end{proposition}

\begin{proof}
Write $n=p^ra$ for a prime $p$, $(p,a) =1$, $r \geq 1$,
then we have
\ben
&& \ord_p \biggl( \frac{1}{m!\cdot n} \biggl( \binom{m n}{n, \dots, n}
-  \binom{m n/p}{n/p, \dots, n/p} \biggr)  \biggr)\\
& = & \ord_p \biggl(\frac{1}{m! p^ra} \frac{(mp^{r-1}a)!}{((p^ra)!)^m}
 \biggl( \frac{(mp^r a)!}{(mp^{r-1} a)!}  - \biggl(\frac{(p^ra)!}{(p^{r-1}a)!} \biggr)^m \biggr) \biggr) \\
& \geq & (mp^{r-1}a + r) + \frac{mp^{r-1}a-S_p(mp^{r-1}a)}{p-1}
- m \frac{p^ra-S_p(p^ra)}{p-1} - \frac{m-S_p(m)}{p-1} - r \\
& = & \frac{1}{p-1} (mS_p(a)- S_p(ma) - m+S_p(m)) \geq 0.
\een
\end{proof}

\subsection{Congruence properties of $\frac{1}{m!} \binom{m(k+l)}{k, \dots, k, l, \dots, l}$}

\begin{proposition} \label{prop:7}
Let $k, l$ be two positive integers,
$p$ a prime such that $(p, k,l) = 1$.
Then for any positive integer $m$,
we have
\be
\frac{1}{m! \cdot (k+l)^m} \frac{(m(k+l))!}{(k!)^m \cdot (l!)^m} \in \bZ_p
\ee
\end{proposition}

\begin{proof}
Because $(p, k, l) = 1$,
we have
$$\min \{\ord_p(k), \ord_p(l)\} = 0.$$
Now we have
\ben
&& \ord_p \biggl( \frac{1}{m! \cdot (k+l)^m} \frac{(m(k+l))!}{(k!)^m \cdot (l!)^m}  \biggr) \\
& = & \frac{1}{p-1} \big[
m(k+l) - S_p(m(k+l)) - m (k-S_p(k)) - m (l- S_p(l)) \\
&& - (m-S_p(m))\big] - m \cdot \ord_p(k+l) \\
& = & \frac{1}{p-1} \big[ m S_p(k) + m S_p(l) - S_p(m(k+l) + S_p(m) - m \big] - m \cdot \ord_p (k+l) \\
& \geq & \frac{1}{p-1}
(mS_p(k) +m S_p(1) - S_p(m) S_p(k+l) + S_p(m) - m )- m \ord_p(k+l) \\
& \geq & \frac{m}{p-1} \big[ S_p(k) + S_p(l) - S_p(k+l)) - (p-1) \cdot \ord_p (k+l)\big]) \\
& + &  (m - S_p(m)) (S_p(k+l) -1) \geq 0.
\een
\end{proof}

\begin{proposition} \label{prop:8}
Let $m$, $k$ and $l$ be positive rational integers,
$k$ and $l$ are divisible by a prime $p$.
Then the following inequality holds:
\be
\ord_p \biggl( \frac{(m(k+l))!}{(m(k+l)/p)!}  - \biggl(\frac{k!l!}{(k/p)!(l/p)!} \biggr)^m \biggr)
\geq m(k+l)/p + \ord_p (k+l),
\ee
whenever $p$ is odd, or $p=2$ and $m$ or $l/2$ is even.
When $p=2$, $m$ and $l/2$ are odd,
the following inequality holds:
\be
\ord_2 \biggl( \frac{(m(k+l))!}{(m(k+l)/2)!}  - \biggl(\frac{k!l!}{(k/2)!(l/2)!} \biggr)^m \biggr)
\geq m(k+l)/2 + 1,
\ee
\end{proposition}

\begin{proof}
By \eqref{eqn:Order} we have:
\ben
&& \ord_p \biggl( \frac{(m(k+l))!}{(m(k+l)/p)!}  \biggr) =  m(k+l)/p.
\een
Write $k+l=p^ra$, $r \geq 1$, $(p, a) = 1$.
Then we have by \eqref{eqn:Remainder}:
\ben
p^{-m(k+l)/p} \frac{(m(k+1))!}{(m(k+l)/p)!} & = & p^{-mp^{r-1}a}  \frac{(mp^r a)!}{(mp^{r-1} a)!}
\equiv (\prod_{\substack{0 < j < p^r \\ (j,p) = 1}} j)^{ma} \pmod{p^r}.
\een
On the other hand,
\ben
\ord_p \biggl(\frac{k!}{(k/p)!} \biggr) = k/p, \qquad
\ord_p \biggl(\frac{l!}{(l/p)!} \biggr) = l/p.
\een
Consider $p^{-l/p} \frac{l!}{(l/p)!}$.
First write it as a product of $l-l/p= (p-1) l/p$ terms,
each term is a number between $l/p=1$ and $l$ with its $p$ factors removed,
then rewrite each term in the form of $p^r - x$.
Modulo $p^r$, one then gets a product of these $x$'s up to a sign of
$(-1)^{(p-1)l/p}$.
It is not hard to see that
\be
p^{-k/p} \cdot \frac{k!}{(k/p)!} \cdot p^{-l/p} \cdot \frac{l!}{(l/p)!}
\equiv (-1)^{(p-1)l/p} (\prod_{\substack{1 \leq j < p^r \\(j,p)=1}} j)^a \pmod{p^r}.
\ee
When $p$ is odd, $p-1$ is even so $ (-1)^{(p-1)l/p}  =1$;
when $p = 2$ and $m$ or $l/2$ is even,
$(-1)^{m(p-1)l/p}=1$.
Therefore, under these conditions,
we have
\be
p^{-m(k+l)/p} \cdot \biggl(\frac{k!l!}{(k/p)!(l/p)!} \biggr)^m
\equiv \biggl(\prod_{\substack{1 \leq j < p^r \\(j,p)=1}} j\biggr)^{ma}  \pmod{p^r}.
\ee
In the case of $p=2$, $l/2$ and $m$ are odd,
\ben
&& 2^{-m(k+l)/2} \biggl(\frac{(m(k+l))!}{(m(k+l)/2)!}  - \biggl(\frac{k!l!}{(k/2)!(l/2)!}  \biggr)^m \biggr)
\equiv 2 \biggl(\prod_{\substack{1 \leq j < 2^r \\(j,2)=1}} j\biggr)^{ma}  \pmod{2^r}.
\een
Therefore the proof is complete.
\end{proof}

\begin{proposition} \label{prop:9}
For any positive integers $m$, $k$ and $l$,
if $k$ and $l$ are divisible by a prime $p$,
then we have:
\be
\frac{1}{m!\cdot (k+l)} \biggl( \frac{(m(k+l))! }{(k!l!)^m}
-  \frac{(m(k+l)/p)!}{((k/p)!(l/p)!)^m} \biggr) \in \bZ_p.
\ee
\end{proposition}

\begin{proof}
When $p$ is odd, or $p=2$ and $m$ or $l/2$ is even,
then we have
\ben
&& \ord_p \biggl( \frac{1}{m!\cdot (k+l)} \biggl( \frac{(m(k+l))! }{(k!l!)^m}
-  \frac{(m(k+l)/p)!}{((k/p)!(l/p)!)^m} \biggr)  \biggr)\\
& = & \ord_p \biggl(  \frac{1}{m!\cdot (k+l)} \cdot \frac{(m(k+l)/p)! }{(k!l!)^m}
 \biggl( \frac{(m(k+l))!}{(m (k+l)/p)!}  - \biggl(\frac{k!l!}{(k/p)!(l/p)!} \biggr)^m \biggr) \biggr) \\
& \geq & (m (k+l)/p + \ord_p(k+l)) + \frac{m(k+l)/p-S_p(m(k+l)/p)}{p-1} \\
&& - m \frac{k-S_p(k)}{p-1}- m \frac{l-S_p(l)}{p-1} - \frac{m-S_p(m)}{p-1} - \ord_p (k+l) \\
& = & \frac{1}{p-1} (mS_p(k) + m S_p(l)- S_p(m(k+l)) - (m-S_p(m))) \\
& \geq & \frac{1}{p-1} (m S_p(k+l) - S_p(m) S_p(k+l) - (m -S_p(m))) \\
& = & \frac{1}{p-1} (m-S_p(m))(S_p(k+l) -1) \geq 0.
\een
When $p=2$ and both $m$ and $l/2$ are odd,
then we have
\ben
&& \ord_2 \biggl( \frac{1}{m!\cdot (k+l)} \biggl( \frac{(m(k+l))! }{(k!l!)^m}
-  \frac{(m(k+l)/2)!}{((k/2)!(l/2)!)^m} \biggr)  \biggr)\\
& = & \ord_2 \biggl(  \frac{1}{m!\cdot (k+l)} \cdot \frac{(m(k+l)/2)! }{(k!l!)^m}
 \biggl( \frac{(m(k+l))!}{(m (k+l)/2)!}  - \biggl(\frac{k!l!}{(k/2)!(l/2)!} \biggr)^m \biggr) \biggr) \\
& \geq & (m (k+l)/2 + 1) + (m(k+l)/2-S_2(m(k+l)/2) ) \\
&& - m (k-S_2(k) )- m (l-S_2(l) ) - (m-S_2(m)) - \ord_2 (k+l) \\
& = & mS_2(k) + m S_2(l)- S_2(m(k+l)) - (m-S_2(m)) - (\ord_2 (k+l)-1) \\
& \geq & m (S_2(k+l) + \ord_2(k+l)-1) - S_2(m) S_2(k+l) \\
&& - (m -S_2(m)) - (\ord_2 (k+l)-1) \\
& = & (m-S_2(m))(S_2(k+l) -1) + (m-1) (\ord_2(k+l) - 1) \geq 0.
\een
\end{proof}

\subsection{Congruence properties of $\binom{m \sum_{i=1}^n k_i}{k_1, \dots, k_1, k_2, \dots, k_2, \dots, k_n, \dots, k_n}$}

Weaker versions of Proposition \ref{prop:Conj2} and Proposition \ref{prop:Conj3}
were conjectured in an earlier version of this paper,
their proofs are communicated to the author by Hanxiong Zhang and Yinhua Ai respectively.
The proof of the present results are only slight modifications of their proofs.

\begin{proposition} \label{prop:Conj2}
For any positive integers $k_1, \dots, k_n$ and $m$,
we have
\be
\frac{(k_1, \dots, k_n)^m}{m! \cdot (\sum_{i=1}^n k_i)^m} \frac{(m\sum_{i=1}^n k_i)!}{\prod_{i=1}^n (k_i!)^m} \in \bZ.
\ee
\end{proposition}

\begin{proof}
For any prime $p$ we have
\ben
&& \ord_p \biggl(\frac{(k_1, \dots, k_n)^m}{m! \cdot (\sum_{i=1}^n k_i)^m} \frac{(m\sum_{i=1}^n k_i)!}{\prod_{i=1}^n (k_i!)^m} \biggr) \\
& = & \frac{m\sum_{i=1}^n k_i - S_p(m\sum_{i=1}^n k_i)}{p-1}- m \frac{\sum_{i=1}^n (k_i - S_p(k_i))}{p-1} \\
&& - \frac{m-S_p(m)}{p-1} - m \ord_p(\sum_{i=1}^n k_i) + m \ord_p ((k_1, \dots, k_n)) \\
& = & \frac{1}{p-1} (  - S_p(m\sum_{i=1}^n k_i) + m S_p(\sum_{i=1}^n k_i)
 - m+S_p(m)) \\
 &&  - m  (\ord_p(\sum_{i=1}^n k_i) - \min\{\ord_p(k_1), \dots, \ord_p(k_n)\}) \\
& \geq & \frac{1}{p-1} ( -S_p(m)S_p(\sum_{i=1}^n k_i) + m \sum_{i=1}^n  S_p(k_i) - m+S_p(m)) \\
& - & m (\ord_p(\sum_{i=1}^n k_i)) - \min \{ \ord_p(k_1), \dots, \ord_p(k_n)\}) \\
& \geq & \frac{1}{p-1} (- S_p(m)S_p(\sum_{i=1}^n k_i) + m  S_p(\sum_{i=1}^n k_i)
 - m+S_p(m)) \\
& = & \frac{1}{p-1} (m -S_p(m)) \cdot (S_p(\sum_{i=1}^n k_i) - 1) \geq 0.
\een
Here in the first inequality we have used Lemma \ref{lm:Multiplicative} and in the second inequality we have
used Lemma \ref{lm:Additive2}.
\end{proof}

\begin{proposition} \label{prop:Conj3}
For any positive integers $m$, $k_1, \dots, k_n$,
then we have:
\be
\frac{1}{p} \cdot \frac{1}{m!\cdot (\sum_{i=1}^n k_i)^m} \biggl( \frac{(m\sum_{i=1}^n k_ip)! }{\prod_{i=1}^n ((k_ip)!)^m}
-  \frac{(m\sum_{i=1}^n k_i)!}{\prod_{i=1}^n (k_i!)^m} \biggr) \in \bZ_p.
\ee
\end{proposition}

\begin{proof}
Note
\ben
&& \frac{1}{p} \cdot \frac{1}{m!\cdot (\sum_{i=1}^n k_i)^m} \biggl( \frac{(m\sum_{i=1}^n k_ip)! }{\prod_{i=1}^n ((k_ip)!)^m}
-  \frac{(m\sum_{i=1}^n k_i)!}{\prod_{i=1}^n (k_i!)^m} \biggr) \\
& = & \frac{1}{p} \cdot  \frac{1}{m!\cdot (\sum_{i=1}^n k_i)^mp} \frac{(m\sum_{i=1}^n k_i)!}{\prod_{i=1}^n (k_i!)^m}
\biggl( \frac{R_p(1) \cdots R_p(m\sum_{i=1}^n k_i)}{\prod_{i=1}^n (R_p(1) \cdots R_p(k_i))^m} - 1 \biggr),
\een
where
\be
R_p(a) = \prod_{j=1}^{p-1} ((a-1)p+j).
\ee
Let $\ord_p(k_1, \dots, k_n) = r$.
By Proposition \ref{prop:Conj2},
\be
\ord_p \biggl( \frac{1}{m!\cdot (\sum_{i=1}^n k_i)^mp} \frac{(m\sum_{i=1}^n k_i)!}{\prod_{i=1}^n (k_i!)^m} \biggr)
\geq - r.
\ee
So it suffices to show that
\be
\ord_p \biggl(\frac{1}{p} \cdot
\biggl( \frac{R_p(1) \cdots R_p(m\sum_{i=1}^n k_i)}{\prod_{i=1}^n (R_p(1) \cdots R_p(k_i))^m} - 1 \biggr)
\biggr) \geq -r,
\ee
or equivalently,
\be
R_p(1) \cdots R_p(m\sum_{i=1}^n k_i)
\equiv \prod_{i=1}^n (R_p(1) \cdots R_p(k_i))^m \pmod{p^{r+1}}.
\ee
Now note if $a \equiv b \pmod{p^r}$,
then one has
\be
R_p(a) \equiv R_p(b) \pmod{p^{r+1}}.
\ee
Now we have $p^r | k_i$,
therefore,
\be
R_p(1) \cdots R_p(k_i) \cong (R_p(1) \cdots R_p(p^r))^{k_i/p^r} \pmod{p^{r+1}},
\ee
and so
\be
\prod_{i=1}^n (R_p(1) \cdots R_p(k_i))^m \cong (R_p(1) \cdots R_p(p^r))^{m\sum_{i=1}^n k_i/p^r} \pmod{p^{r+1}};
\ee
similarly,
\be
R_p(1) \cdots R_p(m \sum_{i=1}^n k_i) \cong (R_p(1) \cdots R_p(p^r))^{m \sum_{i=1}^n k_i/p^r} \pmod{p^{r+1}}.
\ee
This completes the proof.
\end{proof}

\section{Integrality of Some Power Series}

In this section we establish by Dwork's Lemma the integrality of some formal power series
related to local mirror maps,
using the congruence properties of multinomial numbers established in last section.

\subsection{Dwork's Lemma}

Let $p$ be a prime number.

\begin{lemma} \cite{Kob}
Let $F(X) = \sum_i a_iX^i \in 1 + X\bQ_p[[X]]$.
Then $F(X) \in 1 + X\bZ_p[[X]]$ iff $F(X^p)/F(X)^p \in 1 + pX\bZ_p[[X]]$.
\end{lemma}

\begin{lemma} \label{lm:Dwork1}
 \cite{Kob, Lia-Yau} Let $f(X) \in X\bQ_p[[X]]$. Then $e^{f(X)} \in 1 + X\bZ_p[[X]]$ iff
$f(X^p) - pf(X) \in pX\bZ_p[[X]]$.
\end{lemma}

It is straightforward to generalize these lemmas to multivariate case.

\begin{lemma} \cite{Kra-Riv}
Let $F(X_1, \dots, X_n) = \sum_I a_IX^I \in 1 + \sum_{i=1}^n X_i\bQ_p[[X_1, \dots, X_n]]$.
Then $F(X) \in 1 + X\bZ_p[[X]]$ iff $F(X^p)/F(X)^p \in 1 + pX\bZ_p[[X]]$.
\end{lemma}

\begin{lemma} \label{lm:Dwork2}
\cite{Kra-Riv2}
Write $\bX=X_1, \dots, X_n$.
Let $f(\bX) \in \sum_{i=1}^n X_i\bQ_p[[\bX]]$.
Then $e^{f(\bX)} \in 1 + \sum_{i=1}^n X_i\bZ_p[[\bX]]$ iff
$f(\bX^p) - pf(\bX) \in p \sum_{i=1}^n X_i\bZ_p[[\bX]]$,
where $\bX^p =X_1^p, \dots, X_n^p$.
\end{lemma}

\subsection{Integrality of some formal power series}

\begin{theorem} \label{thm:Int1}
Let
$f_m(x) = \frac{1}{m!} \sum_{k=1}^\infty \frac{(mk)!}{(k!)^m} \frac{x^k}{k}$.
Then one has
 $\exp f_m(x) \in 1+x\bZ[[x]]$.
\end{theorem}

\begin{proof}
By Lemma \ref{lm:Dwork1},
it suffices to check that for any prime number $p$,
we have
\ben
&& f_m(x^p) -p f_m(x) \in px \bZ_p[[x]].
\een
Equivalently,
we need to check
\be \label{eqn:1}
\frac{1}{m!k} \binom{mk}{k,\dots,k} \in \bZ_p, \qquad \text{when $(k,p) = 1$},
\ee
and for $(a, p) = 1$, $r \geq 1$,
\be \label{eqn:2}
\frac{1}{m!p^{r-1}a} \binom{mp^ra}{p^ra, \dots, p^ra} -
\frac{1}{m!p^{r-1}a} \binom{mp^{r-1}a}{p^{r-1}a, \dots, p^{r-1}a} \in \bZ_p.
\ee
They are guaranteed by Proposition \ref{prop:4}  and Proposition \ref{prop:6} respectively.
\end{proof}

\begin{theorem} \label{thm:Int2}
Let
$f_m(x_1, x_2) = \frac{1}{m!} \sum_{k_1+k_2 \geq 1} \frac{(m(k_1+k_2))!}{(k_1!)^m(k_2!)^m}
\frac{x_1^{k_1}x_2^{k_2}}{k_1+k_2}$.
Then one has $\exp f_m(x_1, x_2) \in 1+x_1\bZ[[x_1,x_2]] + x_2\bZ[[x_1,x_2]]$.
\end{theorem}

\begin{proof}
It follows from Lemma \ref{lm:Dwork2},
Proposition \ref{prop:7} and Proposition \ref{prop:9}.
\end{proof}

\begin{theorem}
For any positive rational integer $n$,
we have
\be
\exp \bigg( \sum_{k_1+ \cdots +k_n \geq 1} \frac{(m\sum_{i=1}^n k_i)!}{\prod_{i=1}^n (k_i!)^m}
\frac{\prod_{i=1}^n x_i^{k_i}}{\sum_{i=1}^n k_i} \biggr) \in \bZ[[x_1, \dots, x_n]].
\ee
\end{theorem}

\begin{proof}
It follows from Lemma \ref{lm:Dwork2},
Proposition \ref{prop:1} and Proposition \ref{prop:3}.
\end{proof}

\begin{theorem} \label{thm:}
For any positive rational integers $k_1, \dots, k_n$,
we have
\be
\exp \bigg( \sum_{m \geq 1} \frac{(\sum_{i=1}^n k_im)!}{\prod_{i=1}^n (k_im)!}
\frac{x^m}{m} \biggr) \in \bZ[[x]].
\ee
If we have furthermore,
$(k_1, \dots, k_n) = 1$,
then we have
\be
\exp \bigg( \sum_{m \geq 1} \frac{(\sum_{i=1}^n k_im)!}{\prod_{i=1}^n (k_im)!}
\frac{x^m}{\sum_{i=1}^n k_im} \biggr) \in \bZ[[x]].
\ee
\end{theorem}

\begin{proof}
The first assertion follows from Lemma \ref{lm:Dwork1},
Proposition \ref{prop:1.1} and Proposition \ref{prop:3.1}.
The second assertion follows from Lemma \ref{lm:Dwork1}, Proposition \ref{prop:1.2} and Proposition \ref{prop:3.1}.
\end{proof}

\begin{theorem}
For any positive rational integers $m$ and $n$,
we have
\be
\exp \biggl(\frac{1}{m!} \sum_{k_1+ \cdots +k_n \geq 1} \frac{(m\sum_{i=1}^n k_i)!}{\prod_{i=1}^n (k_i!)^m}
\frac{\prod_{i=1}^n x_i^{k_i}}{\sum_{i=1}^n k_i} \biggr) \in \bZ[[x_1, \dots, x_n]].
\ee
\end{theorem}

\begin{proof}
It follows from Lemma \ref{lm:Dwork2},
Proposition \ref{prop:Conj2} and Proposition \ref{prop:Conj3}.
\end{proof}

\section{Integrality of Local Mirror Maps}

In this section we will establish the integrality property of
local mirror map for a noncompact Calabi-Yau $n$-manifold
under suitable conditions on the charge vectors that define it.
The reference for this section is \cite{CKYZ}.

\subsection{Charge vectors for local Calabi-Yau geometries}

Let $M$ be a noncompact toric Calabi-Yau $n$-fold.
It can be obtained by symplectic reduction of some torus action
on an affine space as follows.
Let $(\bC^*)^N$ act on $\bC^{N+n}$ as follows:
\be
(t_1, \dots, t_N) \cdot
(z_0, z_1, \dots, z_{N+n})
= (\prod_{i=1}^N t_i^{l^{(i)}_0} z_0, \dots, \prod_{i=1}^N t_i^{l^{(i)}_{N+n-1}} z_{N+n-1}).
\ee
The vectors
\be
l^{(i)} = (l^{(i)}_0, l^{(i)}_1, \dots, l^{(i)}_{N+n-1}) \in \bZ^{N+n}
\ee
will be referred to as the charge vectors of this torus action.
They satisfy the Calabi-Yau conditions:
\be \label{eqn:CY}
\sum_{j=0}^{N+n-1} l_j^{(i)} = 0, \qquad i =1, \dots, N.
\ee
The noncompact Calabi-Yau $n$-fold $M$ is realized as
$\bC^{N+n}//(\bC^*)^N$.

\subsection{Picard-Fuchs system and local mirror map}

Let $a_0, \dots, a_{N+n-1}$ be some variables and define
\be
z_i = \prod_{j=0}^{N+n-1} a_i^{l^{(i)}_j}, \qquad i = 0, 1, \dots, N.
\ee
They are understood as local coordinates on the  moduli space of complex structure on
the mirror manifold of $M$.
The extended Picard-Fuch system associated to the charge vectors
$l^{(i)}$, $i=1$, $\dots$, $N+n+1$ can be obtained as follows.
Consider the following system of equations:
\be \label{eqn:PF0}
\prod_{l^{(i)}_j > 0} \pd_{a_j}^{l^{(i)}_j} s
= \prod_{l^{(i)}_j<0} \pd_{a_j}^{-l^{(i)}_j} s, \qquad
i =0, \dots, N.
\ee
Assume that $s= s(z_1, \dots, z_N)$.
Then one can rewrite \eqref{eqn:PF0} as
\be \label{eqn:PF}
\prod_{l^{(i)}_j > 0} \prod_{a=1}^{l^{(i)}_j} (\sum_{k=0}^N l^{(k)}_j \theta_k - a) s
= z_i \cdot \prod_{l^{(i)}_j < 0} \prod_{b=0}^{-l^{(i)}_j-1} (\sum_{k=0}^N l^{(k)}_j \theta_k - a) s,
\ee
where $\theta_k = z_k \pd_{z_k}$.
By the Frobenius method,
one gets the following fundamental solution:
\be
\omega = \sum_{m_{1,\dots, N} \geq 0} c(\vec{m}, \vec{r})
 \cdot \prod_{i=1}^N z_i^{r_i+m_i},
\ee
where $\vec{m} = (m_1, \dots, m_N)$, $\vec{r} = (r_1, \dots, r_N)$, and
\be
c(\vec{m}, \vec{r}) =\prod_{j=0}^{N+n-1}
\frac{\Gamma(1+k_j(\vec{l}, \vec{r}))}{\Gamma(1+k_j(\vec{l}, \vec{r}+\vec{m}))}.
\ee
where $k_j(\vec{l}, \vec{m}) = \sum_{i=1}^N l^{(i)}_jm_i$.
Note
\be
\frac{\Gamma(1+x)}{\Gamma(1+m+x)}
= \begin{cases}
1, & m =0, \\
1/\prod_{a=1}^m (a+x), & m > 0, \\
\prod_{b=0}^{-m-1} (-b+x), & m < 0.
\end{cases}
\ee
It follows that
\ben
\frac{\Gamma(1+k_j(\vec{l}, \vec{r}))}{\Gamma(1+k_j(\vec{l}, \vec{r}+\vec{m}))}
= \begin{cases}
1, & k_j(\vec{l}, \vec{m}) =0, \\
1/\prod_{a=1}^{\sum_{i=1}^{N} l^{(i)}_j m_i} (a+k_j(\vec{l}, \vec{r})), & k_j(\vec{l}, \vec{m}) > 0, \\
\prod_{b=0}^{-\sum_{i=1}^{N} l^{(i)}_j m_i-1} (-b+k_j(\vec{l}, \vec{r})), & k_j(\vec{l}, \vec{m}) < 0,
\end{cases}
\een
and so
\ben
&& c(\vec{m}, \vec{r})
= \frac{\prod_{k_j(\vec{l}, \vec{m}) < 0 } \prod_{b=0}^{-k_j(\vec{l}, \vec{m})-1} (-b+k_j(\vec{l}, \vec{r}))}
{\prod_{k_j(\vec{l}, \vec{m}) > 0} \prod_{a=1}^{k_j(\vec{l}, \vec{m})} (a+k_j(\vec{l}, \vec{r}))}.
\een

For $\vec{m} = (m_1, \dots, m_N) \in \bZ_+^{N}$, where
$\bZ_+$ is the set of nonnegative rational integers,
define
\be
N(\vec{l}, \vec{m})
= |\{j \;|\; 0 \leq j \leq N+n-1, k_j(\vec{l}, \vec{m}) < 0 \}|.
\ee
for any positive rational integer $a$.
We will say the charge vectors $\vec{l}$ satisfy Condition (A) if
$N(\vec{l}, \vec{m}) = 0$ implies $\vec{m} = \vec{0}:=(0, \dots, 0)$.
Under this condition one has the following solutions to the Picard-Fuch system \eqref{eqn:PF}:
\bea
&& g_0 = \omega|_{\vec{r}=0} =1, \\
&& g_1^{(i)} = \pd_{r_i} \omega|_{\vec{r}=0} = \log z_i
+ \sum_{N(\vec{l}, \vec{m}) =1} l^{(i)}_{j_{\vec{l}, \vec{m}}}
\frac{(-1)^{k_{j_{\vec{l}, \vec{m}}}(\vec{l}, \vec{m})} (-k_{j_{\vec{l}, \vec{m}}}(\vec{l}, \vec{m}) -1)!}
{\prod_{j \neq j_{\vec{l}, \vec{m}}} k_j(\vec{l}, \vec{m})!} \vec{z}^{\vec{m}},
\eea
where
$\vec{z}^{\vec{m}} = \prod_{i=1}^N z_i^{m_i}$,
and for any $\vec{m}$ with $N(\vec{j}, \vec{m}) = 1$,
$j_{\vec{l}, \vec{m}}$ is the only number $j$ between $0$ and $N+n-1$ such that
$k_{j}(\vec{l}, \vec{m}) < 0$.

Note by the Calabi-Yau condition \eqref{eqn:CY},
one has
\be
-k_{j_{\vec{m}}}(\vec{l}, \vec{m})
= \sum_{j \neq j_{\vec{m}}} k_j(\vec{l}, \vec{m}).
\ee
Also note
\be
k_j(\vec{l}, a \cdot \vec{m}) =a \cdot  k_j(\vec{l}, \vec{m}), \qquad
N(\vec{l}, a \cdot \vec{m}) = N(\vec{l}, \vec{m}), \qquad
j_{\vec{l}, a \cdot \vec{m}} = j_{\vec{l}, \vec{m}},
\ee
for any positive rational integer $a$.
So we can rewrite $g_1^{(i)}$ as follows:
\be \label{eqn:g1(i)}
g_1^{(i)} = \log z_i
+ \sum_{\substack{N(\vec{l}, \vec{m}) =1 \\ \gcd(\vec{m})=1}} l^{(i)}_{j_{\vec{l},\vec{m}}}
\sum_{a=1}^\infty
\frac{(\sum_{j \neq j_{\vec{l},\vec{m}}} k_j(\vec{l}, \vec{m}) \cdot a -1)!}
{\prod_{j \neq j_{\vec{l},\vec{m}}} (k_j(\vec{l}, \vec{m}) \cdot a)!}
((-1)^{k_{j_{\vec{l},\vec{m}}}(\vec{l}, \vec{m})} \vec{z}^{\vec{m}})^a.
\ee
We will say the charge vectors $\vec{l}$ satisfy Condition (B) if
for every $\vec{m}$ such that $N(\vec{l}, \vec{m}) = 1$
and $\gcd(\vec{m}) = 1$ one has
\be
\gcd \{k_j(\vec{l}, \vec{m})\}_{j=0, \dots, N+n-1} = 1.
\ee

The local  mirror map is defined by
\be
q_i = \exp (g_1^{(i)}/g_0), \qquad i =1, \dots, N.
\ee

\begin{theorem}
Let $\vec{l}$ be charge vectors that satisfy Condition (A) and Condition (B),
then the associated local mirror map are given by integral series:
\be
q_i \in z_i \bZ[[z_1, \dots, z_N]], \qquad i=1, \dots, N.
\ee
\end{theorem}

\begin{proof}
This is a straightforward consequence of \eqref{eqn:g1(i)} and Theorem \ref{thm:}.
\end{proof}

\section{Integrality of Local Open-Closed Mirror Maps}

In this section we will establish the integrality property of
the local open-closed mirror map for suitable brane geometry of  noncompact Calabi-Yau $n$-manifold
under suitable conditions on the charge vectors that define it,
based on the results in the preceding two sections.
The references for this section are \cite{May, Ler-May}.

\subsection{Charge vectors for $D$-branes in local Calabi-Yau geometries}

One can also use charge vectors to describe some D-branes in local Calabi-Yau geometries.
The charge vectors
\bea
&& L^{(i)} = (l^{(i)}_0, l^{(i)}_1, \dots, l^{(i)}_{N+n-1},0,0), \qquad i=1,\dots, N, \\
&& L^{(0)} = (1, -1, 0, \dots, 0, -1, 1)
\eea
describe an outer brane in $M$,
while the charge vectors
\bea
&& \tilde{L}^{(i)} = L^{(i)},  \qquad i=2,\dots, N, \\
&& \tilde{L}^{(N)} = L^{(1)} + L^{(0)}, \\
&& \tilde{L}^{(0)} = - L^{(0)} = (-1, 1, 0, \dots, 0, 1, -1)
\eea
describe an inner brane.
One can also consider other charge vectors,
for example,
take any nonempty subset $A$ of $\{1, \dots, N\}$,
define $L_A^{(0)} = - L^{(0)}$ and for $i=1, \dots, N$
\be
L_A^{(i)} = \begin{cases}
L^{(i)}, & i \notin A, \\
L^{(i)} + L^{(0)}, & i \in A.
\end{cases}
\ee
The charge vectors $\vec{L}_A=(L_A^{(0)}, \dots, L_A^{(N)})$
should describe other phases of the inner brane geometry.
The discussions below can be carried out also for such charge vectors.
We leave the details to the reader.

These charge vectors also describe noncompact Calabi-Yau $(n+1)$-folds.
This is a special case of the open-closed string duality \cite{May}.

\subsection{Extended Picard-Fuchs system and open-closed mirror map}

We will write $L^{(i)} = (L^{(i)}_0, \dots, L^{(i)}_{N+n+1})$
and $\tilde{L}^{(i)} = (\tilde{L}^{(i)}_0, \dots, \tilde{L}^{(i)}_{N+n+1})$.
Let $a_0, \dots, a_{N+n+2}$ be some variables and define
\be
Z_i = \prod_{j=0}^{N+n+1} a_i^{L^{(i)}_j}, \qquad i = 0, 1, \dots, N
\ee
and
\be
\tilde{Z}_i = \prod_{j=0}^{N+n+1} a_i^{\tilde{L}^{(i)}_j}, \qquad i = 0, 1, \dots, N.
\ee
They are local coordinates on the D-brane moduli space.
It is clear that
\be
Z_i = z_i, \qquad i = 1, \dots, N.
\ee
Furthermore,
\bea
&& Z_i = \tilde{Z}_i, \qquad i =2, \dots, N, \\
&& Z_0 = \frac{1}{\tilde{Z}_0}, \qquad Z_1 = \tilde{Z}_0 \tilde{Z}_1.
\eea
In the following we will write $Z_i$ as $z_i$ and $\tilde{Z}_i$ as $\tilde{z}_i$.

The extended Picard-Fuch system associated to the charge vectors
$L^{(i)}$, $i=0, 1$, $\dots$, $N+n+1$ can be obtained as follows.
Consider the following system of equations:
\be \label{eqn:PF02}
\prod_{L^{(i)}_j > 0} \pd_{a_j}^{L^{(i)}_j} S
= \prod_{L^{(i)}_j<0} \pd_{a_j}^{-L^{(i)}_j} S, \qquad
i =0, \dots, N.
\ee
Assume that $S= S(z_0, z_1, \dots, z_N)$.
Then one can rewrite \eqref{eqn:PF02} as
\be \label{eqn:PF2}
\prod_{L^{(i)}_j > 0} \prod_{a=1}^{L^{(i)}_j} (\sum_{k=0}^N L^{(k)}_j \theta_k - a) S
= z_i \cdot \prod_{L^{(i)}_j < 0} \prod_{b=0}^{-L^{(i)}_j-1} (\sum_{k=0}^N L^{(k)}_j \theta_k - a) S.
\ee
By the Frobenius method,
one gets the following fundamental solution:
\be
\Omega = \sum_{m_{0,\dots, N} \geq 0} C(\vec{M}, \vec{R})
 \cdot \vec{Z}^{\,\vec{R}+\vec{M}},
\ee
where $\vec{M} = (m_0, m_1, \dots, m_N)$, $\vec{R} = (r_0, r_1, \dots, r_N)$,
$ \vec{Z}^{\,\vec{R}+\vec{M}} = \prod_{i=0}^N z_i^{r_i+m_i}$,
and
\bea
C(\vec{M}, \vec{R})
& = & \prod_{j=0}^{N+n+1}
\frac{\Gamma(1+k_j(\vec{L}, \vec{R}))}{\Gamma(1+k_j(\vec{L}, \vec{R}+\vec{M}))} \\
& = & \frac{\prod_{k_j(\vec{L}, \vec{M}) < 0 }
\prod_{b=0}^{-k_j(\vec{L}, \vec{M})-1} (-b+k_j(\vec{L}, \vec{R}))}
{\prod_{k_j(\vec{L}, \vec{M}) > 0} \prod_{a=1}^{k_j(\vec{L}, \vec{M})}
(a+k_j(\vec{L}, \vec{R}))},
\eea
where $k_j(\vec{L}, \vec{M}) = \sum_{i=0}^N L^{(i)}_jm_i$.

Assume that $\vec{L}$ satisfies Condition (A).
Then one has the following solutions to the Picard-Fuch system \eqref{eqn:PF2}:
\ben
&& G_0 = \Omega|_{\vec{R}=0} =1, \\
&& G_1^{(i)} = \pd_{r_i} \Omega|_{\vec{R}=0} = \log z_i
+ \sum_{N(\vec{L}, \vec{M}) =1} l^{(i)}_{j_{\vec{L}, \vec{M}}}
\frac{(-1)^{k_{j_{\vec{L}, \vec{M}}}(\vec{L}, \vec{M})} (-k_{j_{\vec{L}, \vec{M}}}(\vec{L}, \vec{M}) -1)!}
{\prod_{j \neq j_{\vec{L}, \vec{M}}} k_j(\vec{L}, \vec{M})!} \vec{Z}^{\vec{M}}.
\een
The local open-closed mirror map is defined by
\be
Q_i = \exp (G_1^{(i)}/G_0), \qquad i =0, 1, \dots, N.
\ee
Similar things can be done  the inner brane case.
It is easy to see that
\be
Q_0 = \frac{1}{\tilde{Q}_0}, \qquad Q_1 = \tilde{Q}_0 \tilde{Q}_1, \qquad
Q_i = \tilde{Q}_i, i =2, \dots, N.
\ee

\begin{theorem} \label{thm:Mirror}
If the charge vectors $\vec{L}$ describing an outer brane geometry in a toric Calabi-Yau $n$-fold
satisfy Condition (A) and Condition (B),
then the associated local open-closed mirror map are given by integral series:
\be
Q_i \in z_i \bZ[[z_0,z_1, \dots, z_N]], \qquad i=0,1, \dots, N.
\ee
Similarly,
if the charge vectors $\vec{\tilde{L}}$ describing an inner brane geometry in a toric Calabi-Yau $n$-fold
satisfy Condition (A) and Condition (B),
then the associated local open-closed mirror map are given by integral series:
\be
\tilde{Q}_i \in \tilde{z}_i \bZ[[\tilde{z}_0, \tilde{z}_1, \dots, \tilde{z}_N]], \qquad i=0,1, \dots, N.
\ee
\end{theorem}

For examples,
see \cite[Appendix A]{Ler-May}.

\subsection{Integrality properties of the inverse series}

We will also establish the integrality property of the inverse local open-closed  mirror maps
by the following easy observation:

\begin{lemma} \label{lm:Inverse1}
Suppose that
\be
Z_i = z_i + \sum_{m_0+ \cdots +m_n > 1} a^{(i)}_{m_0, \dots, m_n} z_0^{m_1} \cdots z_n^{m_n}
\ee
are formal power series in $\bZ[[z_0, \dots, z_n]]$
for $i=0, \dots, n$.
Then one can find their formal inverse series of the form:
\be
z_i = Z_i + \sum_{m_0+ \cdots +m_n > 1} b^{(i)}_{m_0, \dots, m_n} Z_0^{m_1} \cdots Z_n^{m_n}
\ee
in $\bZ[[z_0, \dots, z_n]]$
for $i=0, \dots, n$.
\end{lemma}

Furthermore,
one can find the inverse series explicitly using the following:

\begin{lemma} \label{lm:Inverse2}
Suppose that for $i=0, 1, \dots, n$,
$Z_i =z_i e^{f_i(z_0, \dots, z_n)}$
for some analytic function $f_i(z_0, \dots, z_n)$ such that $f_i(0, \dots, 0) = 0$.
Then there is an inverse map of the form
$z_i =Z_i + \sum_{m_0+ \cdots +m_n > 0} b^{(i)}_{m_0, \dots, m_n} Z_0^{m_0} \cdots Z_n^{m_n}$,
where
$b^{(i)}_{m_0, \dots, m_n}$
is the coefficient of $\prod_{j=0}^n z_i^{m_j-\delta_{ij}}$ in
\be \label{eqn:Inverse}
 \exp (-\sum_{j=0}^n m_j f_j(z_0, \dots, z_n))
 \cdot \begin{vmatrix}
1 + \theta_0 f_0 & \theta_0 f_1 & \cdots & \theta_0 f_n \\
\theta_1 f_0 & 1 + \theta_1 f_1 & \cdots & \theta_1 f_n \\
\vdots & \vdots & & \vdots \\
\theta_n f_0 & \theta_n f_1 & \cdots & 1+ \theta_n f_n
\end{vmatrix}.
\ee
\end{lemma}

\begin{proof}
This is a special case of the Lagrange-Good inversion formula \cite{Good}.
First of all,
we have
\ben
&& dZ_0 \wedge \cdots \wedge dZ_n \\
& = & e^{\sum_{j=0}^n f_j(z_0,\dots, z_n)}
 \begin{vmatrix}
1 + \theta_0 f_0 & \theta_0 f_1 & \cdots & \theta_0 f_n \\
\theta_1 f_0 & 1 + \theta_1 f_1 & \cdots & \theta_1 f_n \\
\vdots & \vdots & & \vdots \\
\theta_n f_0 & \theta_n f_1 & \cdots & 1+ \theta_n f_n
\end{vmatrix}
\cdot dz_0 \wedge \cdots \wedge dz_n.
\een
Therefore,
\ben
&& b_{m_0, \dots, m_n}^{(i)}
= \res_{Z_0=0} \cdot \res_{Z_n=0} \frac{z_i}{Z_0^{m_0+1} \cdots Z_n^{m_n+1}} \\
& = & \frac{1}{(2\pi i)^{n+1}} \oint \cdots \oint \frac{z_i}{\prod_{j=0}^n Z_j^{m_j+1}} dZ_0 \wedge \cdots \wedge dZ_n \\
& = & \frac{1}{(2\pi i)^{n+1}} \oint \cdots \oint
\frac{1}{\prod_{j=0}^n z_j^{m_j+1-\delta_{i,j}}} \exp (-\sum_{j=0}^n m_j f_j(z_0, \dots, z_n)) \\
&& \cdot  \begin{vmatrix}
1 + \theta_0 f_0 & \theta_0 f_1 & \cdots & \theta_0 f_n \\
\theta_1 f_0 & 1 + \theta_1 f_1 & \cdots & \theta_1 f_n \\
\vdots & \vdots & & \vdots \\
\theta_n f_0 & \theta_n f_1 & \cdots & 1+ \theta_n f_n
\end{vmatrix}
\cdot dz_0 \wedge \cdots \wedge dz_n.
\een
\end{proof}

\begin{remark} \label{rmk:Inverse}
A particular case is when $f_i$'s are only functions of $z_1, \dots, z_n$.
Then \eqref{eqn:Inverse} becomes
\be
 \exp (-\sum_{j=0}^n m_j f_j(z_1, \dots, z_n))
 \cdot \begin{vmatrix}
1+ \theta_1 f_1 & \theta_{z_1} f_2 & \cdots & \theta_1 f_n \\
\vdots & \vdots & & \vdots \\
\theta_n f_1 & \theta_n f_2 & \cdots & 1+ \theta_n f_n
\end{vmatrix}.
\ee
If follows that for $i=1, \dots, n$, $b^{(i)}_{m_0, \dots, m_n}$
is $\delta_{m_0,0}$ times the coefficient of $\prod_{j=1}^n z_i^{m_j-\delta_{ij}}$ in
\be
 \exp (-\sum_{j=1}^n m_j f_j )
 \cdot \begin{vmatrix}
1+ \theta_1 f_1 & \theta_{z_1} f_2 & \cdots & \theta_1 f_n \\
\vdots & \vdots & & \vdots \\
\theta_n f_1 & \theta_n f_2 & \cdots & 1+ \theta_n f_n
\end{vmatrix},
\ee
and
$b^{(0)}_{m_0, \dots, m_n}$ is $\delta_{m_0,1}$ times the coefficient of $z_1^{m_1} \cdots z_n^{m_n}$
in
\be
 \exp (-f_0 - \sum_{j=1}^n m_j f_j )
 \cdot \begin{vmatrix}
1+ \theta_1 f_1 & \theta_{z_1} f_2 & \cdots & \theta_1 f_n \\
\vdots & \vdots & & \vdots \\
\theta_n f_1 & \theta_n f_2 & \cdots & 1+ \theta_n f_n
\end{vmatrix}.
\ee
\end{remark}

\subsection{Superpotential functions and integrality properties of the local mirror curve}

In the outer brane case described by charge vectors $\vec{L}$,
take the superpotential to be the following part of
$\half \pd_{r_0}^2\Omega|_{\vec{R}=0}$:
\be
W = \sum_{\substack{m_0>0, k_1(\vec{L}, \vec{M}) <0 \\k_j(\vec{L}, \vec{M})\geq 0, j=0,2, \dots, N+n-1}}
\frac{(-1)^{k_1(\vec{L}, \vec{M})+m_0} (-k_1(\vec{L}, \vec{M})-1)!}{m_0 \cdot
\prod_{0 \leq j \leq N+n-1, j\neq 1} k_j(\vec{L}, \vec{M})!} \cdot \vec{Z}^{\,\vec{M}}.
\ee
The $n=3$ case can be found in \cite[(3.9)]{Ler-May}.
Similarly,
in the inner brane case described by charge vectors $\vec{L}$,
take the superpotential to be following part of
$\half \pd_{r_0}^2 \tilde{\Omega}|_{\vec{R}=0}$:
\bea
&& \tilde{W} = \sum_{\substack{m_0 \neq m_1, k_0(\vec{\tilde{L}}, \vec{M}) < 0 \\k_j(\vec{\tilde{L}}, \vec{M})\geq 0, j=1, \dots, N+n-1}}
\frac{(-1)^{k_1(\vec{\tilde{L}}, \vec{M})+m_0-m_1} (-k_0(\vec{\tilde{L}}, \vec{M})-1)!}{(m_0-m_1) \cdot
\prod_{j=1}^{N+n-1} k_j(\vec{\tilde{L}}, \vec{M})!} \cdot \vec{\tilde{Z}}^{\,\vec{M}}.
\eea
Examples of the $n=3$ case can be found in  \cite[Apeendix A]{Ler-May}.
Note we have
\ben
&& \theta_0 W = \sum_{\substack{m_0>0, k_1(\vec{L}, \vec{M}) <0 \\k_j(\vec{L}, \vec{M})\geq 0, j=0,2, \dots, N+n-1}}
\frac{(-1)^{k_1(\vec{L}, \vec{M})+m_0} (-k_1(\vec{L}, \vec{M})-1)!}
{\prod_{0 \leq j \leq N+n-1, j\neq 1} k_j(\vec{L}, \vec{M})!} \cdot \vec{Z}^{\,\vec{M}}.
\een
\ben
&& (\tilde{\theta}_0 - \tilde{\theta}_1) \tilde{W}
 = \sum_{\substack{m_0 \neq m_1, k_0(\vec{\tilde{L}}, \vec{M}) < 0 \\k_j(\vec{\tilde{L}}, \vec{M})\geq 0, j=1, \dots, N+n-1}}
\frac{(-1)^{k_1(\vec{\tilde{L}}, \vec{M})+m_0-m_1} (-k_0(\vec{\tilde{L}}, \vec{M})-1)!}
{\prod_{j=1}^{N+n-1} k_j(\vec{\tilde{L}}, \vec{M})!} \cdot \vec{\tilde{Z}}^{\,\vec{M}}.
\een
Therefore,
similar to Theorem \ref{thm:Mirror}
we have:

\begin{theorem}
If the charge vectors $\vec{L}$ describing an outer brane geometry in a toric Calabi-Yau $n$-fold,
then one has:
\be
\exp (- \theta_0 W) \in \bZ[[z_0,z_1, \dots, z_N]].
\ee
Similarly,
if the charge vectors $\vec{\tilde{L}}$ describing an inner brane geometry in a toric Calabi-Yau $n$-fold,
then one has:
\be
\exp(- (\tilde{\theta}_0-\tilde{\theta}_1) \tilde{W}) \in \bZ[[\tilde{z}_0, \tilde{z}_1, \dots, \tilde{z}_N]].
\ee
\end{theorem}

For examples,
see \cite[Appendix A]{Ler-May}.
According to \cite{AV, AKV},
$y= \exp (-\theta_0 Q)$ gives the equation of the local mirror curve.
We will clarify the geometric meaning of $\exp(- (\tilde{\theta}_0-\tilde{\theta}_1) \tilde{W})$
in a forthcoming paper \cite{Zho2}.

\end{document}